\documentclass[a4paper, 12pt, twoside, onecolumn ]{article}
\usepackage[utf8]{inputenc}
\usepackage{amsmath, amssymb, amsthm,mathrsfs}
\usepackage{enumerate} 
\usepackage{graphicx}
\usepackage[top=1in, bottom=1.3in, left=1.3in, right=1in]{geometry}

\usepackage{lmodern}
\usepackage[all]{xy}
\usepackage{tikz-cd} 
 \usepackage{float}
 \usepackage{hyperref}
 \usepackage{accents}
 \usepackage{verbatim}
 \usepackage{import}
 \usepackage{multirow}
 \usepackage{multicol}
\date{}

\newtheorem{defn}{\bf Definition}[section] 

\newtheorem{lem}[defn]{\bf Lemma}
\newtheorem{thm}[defn]{\bf Theorem} 

\newtheorem{rmk}{\bf Remark}

\numberwithin{equation}{section}
\usepackage[pagewise]{lineno}

\newcommand{\footremember}[2]{%
    \footnote{#2}
    \newcounter{#1}
    \setcounter{#1}{\value{footnote}}%
}
 
\title{  Error Estimates for a Linearized Fractional Crank-Nicolson FEM  for Kirchhoff type  Quasilinear Subdiffusion Equation with Memory }
\author{%
  Lalit Kumar \footremember{alley}{Department of Mathematics, Indian Institute of Technology Bombay, Mumbai-400076, India. lalitccc528@gmail.com}%
  }
\date{}
\begin{document}
\maketitle

\begin{abstract}
\noindent  In this paper, we develop  a  linearized fractional Crank-Nicolson-Galerkin FEM for Kirchhoff type  quasilinear time-fractional integro-differential equation $\left(\mathcal{D}^{\alpha}\right)$. In general, the solutions of the time-fractional problems  exhibit a weak singularity at  time $t=0$.  This singular behaviour of the solutions is taken into account while deriving the convergence estimates of the developed numerical scheme. We prove that  the proposed  numerical scheme has an  accuracy rate of  $O(M^{-1}+N^{-2})$ in $L^{\infty}(0,T;L^{2}(\Omega))$ as well as in $L^{\infty}(0,T;H^{1}_{0}(\Omega))$, where $M$ and $N$ are the degrees of freedom in the space and time directions respectively.  A numerical experiment is presented to verify the  theoretical results. 

\end{abstract}



\section{Introduction}
The goal of this article is to study the following time-fractional PDE with Kirchhoff type nonlocal diffusion coefficient involving memory/feedback
term.\\

\noindent Let $\Omega$ be a convex and bounded subset of $\mathbb{R}^{d}~(d\geq 1)$ with smooth boundary $\partial \Omega$ and $[0,T]$ is a fixed finite time interval. 
 Find the unknown function $u:=u(x,t) :\Omega \times [0,T]\rightarrow \mathbb{R} $ that satisfies 
\begin{equation} \tag{$\mathcal{D}^{\alpha}$} \label{22-5-1}
\begin{aligned}
^{C}\mathcal{D}^{\alpha}_{t}u-\left(1+\int_{\Omega}|\nabla u|^{2}~dx\right)\Delta u&=f(x,t)-\int_{0}^{t}\Delta u(s)~ds ~\text{in}~ \Omega \times (0,T],\\
u(x,0)&=u_{0}(x) \quad \text{in} ~\Omega,\\
u(x,t)&=0 \quad \text{on} ~\partial \Omega \times [0,T],
\end{aligned}
\end{equation}
where initial value $u_{0}$ and   source term $f$ are known. In the equation \eqref{22-5-1}, the notation $^{C}\mathcal{D}_{t}^{\alpha}u$
 denotes the   time-fractional derivative of order $\alpha,~(0<\alpha<1)$ in the Caputo sense \cite{podlubny1998fractional} as
\begin{equation}\label{LS1}
^{C}\mathcal{D}_{t}^{\alpha}u =\frac{1}{\Gamma(1-\alpha)} \int_{0}^{t}\frac{1}{(t-s)^{\alpha}}\frac{\partial u}{\partial s}(s)ds,
\end{equation}
where $\Gamma(\cdot)$ denotes the gamma function.\\

 \noindent In the past decade, researchers are interested to the study differential equations which involve fractional derivatives. The nonlocal property of these fractional derivatives provides a powerful mathematical tool to describe many physical  phenomena accurately and in a realistic manner \cite{almeida2016modeling,luchko2012anomalous}. Therefore, fractional PDEs are being widely  investigated in theory as well as in  numerical perspective \cite{chen2019error,kopteva2019error,lalit,ren2021sharp} nowadays.\\
 
 \noindent There are many physical or biological  processes in which quantity of interest depends on the whole space domain instead of pointwise. For example, displacement of a vibrating string in which length of the string changes during vibrations, diffusion of a bacteria in a jar and many more. These type of phenomena are driven by quasilinear PDEs which contain the Kirchhoff type diffusion coefficient \cite{almeida2017finite,chipot2003remarks,gudi2012finite,kundu2016kirchhoff}.\\
 
 \noindent In real life situations, most of the phenomena  happen in a non-homogeneous medium for which background of the process is required. For example heat conduction materials with memory, transport in heterogeneous media. These types of processes are  modelled by integro-differential equations \cite{ferreira2007memory, kumar2020finite, mahata2021finite,miller1978integrodifferential}.  With the help of  mathematical formulation \eqref{22-5-1}, we analyze the model in which above mentioned  phenomena occur simultaneously.\\
 
 \noindent For the standard diffusion case $(\alpha=1)$ in \eqref{22-5-1}, Lalit et al \cite{kumar2020finite} developed a linearized Crank-Nicolson-Galekin FEM which has second order accuracy in the time direction. Same authors in  \cite{lalit} studied the subdiffusion case \eqref{22-5-1} and proposed a linearized fractional Crank-Nicolson-Galerkin FEM which is also second order accurate in the time direction.   In \cite{lalit}, uniform mesh is used in the time direction which works perfectly well for the sufficiently smooth $(u \in C^{4}[0,T])$ solutions of the problem \eqref{22-5-1}.\\
 
 \noindent In general, the solutions of the time-fractional PDEs have the limited smoothness even if the data is sufficiently smooth \cite{jin2018numerical,jin2019subdiffusion,sakamoto2011initial,stynes2017error}.  This phenomenon makes the time-fractional PDEs more challenging to study  than the standard diffusion case. In \cite{stynes2017error}, Stynes et al proved that the solution  of the following linear time-fractional PDE  
 \begin{equation} \label{22-8-22-1}
\begin{aligned}
^{C}\mathcal{D}^{\alpha}_{t}u-u_{xx}&=f(x,t)\quad \text{in}~ (0,1) \times (0,T],\\
u(x,0)&=u_{0}(x) \quad \text{in} ~[0,1],\\
u(0,t)=u(1,t)&=0 \quad \text{in}~ [0,T],
\end{aligned}
\end{equation}
satisfies 
\begin{equation}\label{23-6-3ab}
    |\partial^{l}_{t}u|\lesssim C\left(1+t^{\alpha-l}\right),~~l=0,1,2,3,
\end{equation}
for some positive constant $C$ which is independent of $t$ but may depend on $T$. Further Jin et al in \cite{jin2019subdiffusion} studied the following linear time-fractional PDE with time dependent diffusion coefficient 
 
 \begin{equation} \label{22-8-22-3}
\begin{aligned}
^{C}\mathcal{D}^{\alpha}_{t}u-\nabla\cdot (a(x,t))\nabla u(x,t))&=f(x,t)\quad \text{in}~ \Omega  \times (0,T],\\
u(x,0)&=u_{0}(x) \quad \text{in} ~\Omega,\\
u(x,t)&=0 \quad \text{on}~ \partial \Omega \times [0,T].
\end{aligned}
\end{equation}
They have proved that the solution $u$ of \eqref{22-8-22-3} follows $ \|\partial_{t}u\|\lesssim C~t^{\alpha-1}$ for smooth initial data $u_{0}(x)$.\\

\noindent We  derive the convergence estimates for the problem \eqref{22-5-1} under the assumption that the solution $u$ of \eqref{22-5-1} satisfies \eqref{23-6-3ab}. Although, to prove the estimate \eqref{23-6-3ab} for the solutions of \eqref{22-5-1}  is an open problem till now. For the  solutions of the type \eqref{23-6-3ab}, we have a loss of accuracy if we use a uniform mesh in the time direction (see Table  \ref{table2}, Section \ref{26-6-1}). In this work,  we recover this loss of accuracy by developing a  numerical scheme on a non-uniform time mesh in  the time direction for the non-smooth solutions \eqref{23-6-3ab} of the problem \eqref{22-5-1}.  \\

 \noindent The lack of smoothness \eqref{23-6-3ab} of the solutions of \eqref{22-5-1} is compensated by taking the graded mesh in the time direction, which are concentrated near $t=0$ (see Figure 1, Section \ref{26-6-1} ) \cite{ren2021sharp,stynes2017error}. These graded mesh on $[0,T]$ are defined by 
 \begin{equation}\label{23-5-1c}
    t_{n}=T\left(\frac{n}{N}\right)^{r},~n=0,1,\dots,N,
\end{equation}
where $r\geq 1$  is called grading exponent and $N$ is any positive integer.\\

\noindent In this work, we approximate the Caputo fractional derivative \eqref{LS1} using Alikhanov's L2-1$_{\sigma},~(0<\sigma<1)$ scheme \cite{alikhanov2015new} on the time graded mesh \eqref{23-5-1c}. This scheme approximate the Caputo fractional derivative  at $t_{n-\sigma} \in [0,T]$ and   carry out a local truncation error of $O\left(t_{n-\sigma}^{-\alpha}~N^{-\min\{r\alpha,~3-\alpha\}}\right)$ \cite{chen2019error}.\\

 \noindent The presence of the nonlocal nonlinearity $\left(1+\int_{\Omega}|\nabla u|^{2}~dx\right)$ in the problem \eqref{22-5-1} creates a major  difficulty to evaluate the numerical solutions of  \eqref{22-5-1}. It can be seen through the non-sparsity of the  Jacobian of the corresponding nonlinear system \cite{gudi2012finite}.  For solving the nonlinear system with this non-sparse Jacobian, the Newton-Raphson  method demands high computational cost and huge computer storage. We encounter this difficulty by developing a linearized approximation \eqref{25-5-3} of nonlinear Kirchhoff term at $t_{n-\sigma}$ on the  time graded mesh. We show that this approximation  commits a linearization error of $O\left(N^{-\min\{r\alpha,~2 \}}\right)$.\\
 
 \noindent Further complexity of the problem \eqref{22-5-1} is to discretize the memory term $\left(\int_{0}^{t}\Delta u(s)~ds\right)$ in such a way that  takes into account of weak singularity of the solution $u$  at $t=0$ and remains consistent with the linearization scheme \eqref{25-5-3}. To approximate the memory term on $[0,t_{n-\sigma}]$, we break the  time interval $[0,t_{n-\sigma}]$ into $[0,t_{1}]$ and  $[t_{1},t_{n-\sigma}]$.  In the first time interval $[0,t_{1}]$ we approximate the memory term on  $[0,t_{1-\sigma}]$ using right rectangle rule \eqref{27-5-5}. The second interval  $[t_{1},t_{n-\sigma}]$ is divided into two parts $[t_{1},t_{n-1}]$ and $[t_{n-1},t_{n-\sigma}]$, in the first part we apply composite trapezoidal rule and in the later part we use left rectangle rule.  For this approximation, we prove that the quadrature error is of $O\left(N^{-\min\{r(1+\alpha),~2\}}\right)$ on the time graded mesh.\\
 
 \noindent  Based on the above mentioned approximations, we construct a fully discrete formulation \eqref{27-5-9a}-\eqref{27-5-9b} of the problem \eqref{22-5-1} by discretizing the space domain $\Omega$. Traditionally, fractional Gronwall's inequality \cite{liao2019discrete,ren2021sharp} is applied to conclude the global rate of convergence for the time-fractional problems. But, these types of inequalities are no more applicable in our problem \eqref{22-5-1} due to  the  presence of the memory term. We overcome this difficulty by presenting  a  new approach in which we use a new weighted $H^{1}(\Omega)$ norm \eqref{22-6-4}. We prove that our developed numerical scheme has an optimal convergence rate of  $O(M^{-1}+N^{-2})$, under suitably chosen grading parameter $(r)$ in the graded mesh. This  result is supported by conducting a numerical experiment.   In this way, this is  the first work in   the literature which  contributes a robust and efficient second-order (in time ) accurate  linearized numerical scheme for the Kirchhoff type time-fractional  problems with memory operator.  \\

   \noindent \textbf{This article is structured as follows.} Section 2 contains some prelimimaries concepts. In Section 3, we prove  convergence analysis of the  developed linearized fractional Crank-Nicolson-Galerkin FEM. Numerical implementation in Section 4 confirms the reliability of the obtained theoretical results.  In  Section 5, we  present some concluding remarks. \\

\section{Preliminaries}
In this section, we set up some notations and state some preliminary results.

\noindent \textbf{Notations:} Let $L^{2}(\Omega)$ be the space of square integrable functions with the norm $\|\cdot\|$ which is induced by the inner product $(\cdot, \cdot)$. Denote $H^{m}(\Omega)$ by the standard Sobolev spaces with the norm $\|\cdot\|_{m}$, where $m$ is any positive number. Set $H^{1}_{0}(\Omega)=\left\{ u \in  H^{1}(\Omega)~:~ u=0~\text{on}~\partial \Omega\right\}$.

\noindent For any Hilbert space $X$, the space $L^{2}(0,T;X)$ consists of all measurable functions $\phi : [0,T] \rightarrow X$ with the norm 
\begin{equation}
    \|\phi\|_{L^{2}(0,T;X)}^{2}= \int_{0}^{T}\|\phi(t)\|^{2}_{X}~dt < \infty,
\end{equation}
and the norm on $L^{2}_{\alpha}(0,T;X)$ is given by 
\begin{equation}
    \|\phi\|_{L^{2}_{\alpha}(0,T;X)}^{2}= \sup_{t\in (0,T)}\left(\frac{1}{\Gamma(\alpha)}\int_{0}^{t}(t-s)^{\alpha-1}\|\phi(s)\|^{2}_{X}~dt\right) < \infty.
\end{equation}
The space $L^{\infty}(0,T;X)$ consists of all measurable functions $\phi: [0,T] \rightarrow X$ with the norm 
\begin{equation}
    \|\phi\|_{L^{\infty}(0,T;X)}=\text{ess}\sup_{t\in (0,T) }\|\phi(t)\|_{X} < \infty.
\end{equation}
\noindent For any two quantities $a$ and $b$, the notation $a\lesssim b$ means that there exists a generic positive constant $C$ such that $a \leq Cb$, where $C$ is independent of discretization parameters.\\

\noindent \textbf{Weak  formulation:} The weak formulation of  the problem \eqref{22-5-1} is to seek $u(t)\in H^{1}_{0}(\Omega),~t\in [0,T]$ such that following equations hold for all $v \in H^{1}_{0}(\Omega)$
\begin{equation} \tag{$\mathcal{W}^{\alpha}$} \label{22-5-2}
\begin{aligned}
\left(^{C}\mathcal{D}^{\alpha}_{t}u,v\right)+\left(1+\|\nabla u(t)\|^{2}\right)(\nabla u, \nabla v)&=
(f,v)+\int_{0}^{t}(\nabla  u(s),\nabla v)~ds ~\text{a.e.}~t \in (0,T],\\
u(x,0)&=u_{0}(x) \quad \text{in} ~\Omega.\\
\end{aligned}
\end{equation}
\noindent The following theorem establishes the existence, uniqueness, and a priori bounds on the solution of \eqref{22-5-2}.
\begin{thm}\label{23-6-1a}\cite{lalit,kundu2016kirchhoff} Suppose that  $u_{0}\in H^{1}_{0}(\Omega)$ and $f \in L^{\infty}(0,T;L^{2}(\Omega))$. Then there exists a unique solution to the problem \eqref{22-5-2}. Further the solution $u$ of  \eqref{22-5-2} satisfies the following a priori bounds which are valid uniformly in time 
\begin{equation}\label{23-6-1}
\|u\|_{L^{\infty}(0,T;L^{2}(\Omega))}+\|u\|_{L^{2}_{\alpha}(0,T;H^{1}_{0}(\Omega))}\lesssim \|u_{0}\|+\|f\|_{L^{\infty}(0,T;L^{2}(\Omega))},
\end{equation}
and 
\begin{equation}\label{23-6-2}
\|u\|_{L^{\infty}(0,T;H^{1}_{0}(\Omega))}+\|u\|_{L^{2}_{\alpha}(0,T;H^{2}(\Omega))}\lesssim \|u_{0}\|_{1}+\|f\|_{L^{\infty}(0,T;L^{2}(\Omega))}:=K.
\end{equation}
\end{thm}

\noindent \textbf{Semi discrete formulation:} For the semi discrete formulation of the problem \eqref{22-5-1}, we keep the time variable continuous and discretize the space domain $\Omega$. The space domain $\Omega$ is divided into a quasi-uniform triangulation using a conforming FEM \cite{thomee2007galerkin}. We consider a finite element subspace $\mathcal{V}_{h}$ of $H^{1}_{0}(\Omega)$ which consists of piece-wise continuous  linear functions on each triangle of the triangulation.\\

\noindent Find $u_{h}\in \mathcal{V}_{h}$ such that the following equations holds for all $v_{h}\in \mathcal{V}_{h}$ and a.e. $t \in [0,T]$
\begin{equation} \tag{$\mathcal{S}^{\alpha}$} \label{23-6-3}
\begin{aligned}
\left(^{C}\mathcal{D}^{\alpha}_{t}u_{h},v_{h}\right)+\left(1+\|\nabla u_{h}(t)\|^{2}\right)(\nabla u_{h}, \nabla v_{h})&=
(f_{h},v_{h})+\int_{0}^{t}(\nabla  u_{h}(s),\nabla v_{h})~ds,\\
u_{h}(x,0)&=R_{h}u_{0} \quad \text{in} ~\Omega,\\
\end{aligned}
\end{equation}
where $f_{h}$ is the $L^{2}$-projection of $f$ into $\mathcal{V}_{h}$ and $R_{h}u_{0}$ is the Ritz-projection of $u_{0}$ into $\mathcal{V}_{h}$ \cite{thomee2007galerkin}.\\

\noindent The existence, uniqueness and a priori bounds  on  the semi discrete solution $u_{h}$ of \eqref{23-6-3} are  established along the same lines of the proof of Theorem \ref{23-6-1a}. For further  analysis, we assume that the solution $u$ of the problem \eqref{22-5-1} satisfies the following realistic regularity \cite{jin2018numerical,jin2019subdiffusion,stynes2017error}
\begin{equation}\label{28-6-1}
    \|\partial^{l}_{t}u\|_{p}\lesssim t^{\alpha-l},~~l=1,2,3,~p=0,1,2.
\end{equation}

\noindent  To derive the semi discrete error analysis, we modify the  standard Ritz-Volterra projection operator \cite{cannon1990priori} of $u$ in such a way that it removes the difficulty  arising from the nonlocal nonlinear diffusion coefficient. Modified Ritz-Volterra  projection operator $w:[0,T]\rightarrow \mathcal{V}_{h}$ of $u$ satisfies
\begin{equation}\label{23-6-4}
    \left(1+\|\nabla u(t)\|^{2}\right)(\nabla (u-w), \nabla v_{h})=\int_{0}^{t}(\nabla  (u-w)(s),\nabla v_{h})~ds~~\forall~v_{h}\in \mathcal{V}_{h}.
\end{equation}

\noindent It can be seen that \eqref{23-6-4} is a system of Volterra type integro-differential equations. Due to the positivity of diffusion coefficient and linearity of \eqref{23-6-4} ensure the existence and uniqueness of $w$ \cite{cannon1990priori}.\\

\noindent Using the idea of the proof of Lemma 3.3 and Lemma 3.4 from  \cite{kumar2020finite}, one can derive the following stability results on the modified Ritz-Volterra projection operator
\begin{equation}\label{23-6-5aa}
    \|\nabla w\| \lesssim K\quad \text{and}\quad  \|\Delta_{h} w\| \lesssim K,
\end{equation}
where $\Delta_{h}w$ is the discrete Laplacian operator defined by
 $\Delta_{h}:\mathcal{V}_{h}\rightarrow \mathcal{V}_{h}$ such that 
     \begin{equation}\label{8-7-1w}
         (-\Delta_{h}u_{h},v_{h})=(\nabla u_{h},\nabla v_{h})~\forall~u_{h},v_{h} \in \mathcal{V}_{h}.
     \end{equation}

\begin{thm}\cite{cannon1990priori}\label{23-6-5} Suppose that the solution $u$ of the equation $(\mathcal{D}^{\alpha})$ satisfies regularity estimate  \eqref{28-6-1}. Then the modified Ritz-Volterra projection operator $w$ of $u$ has  the following  best approximation properties 

\begin{equation}\label{23-6-5a}
    \|u-w\|+h\|\nabla(u-w)\|\lesssim~h^{2},
    \end{equation}
    and 
\begin{equation}\label{23-6-7}
     \|~^{C}D_{t}^{\alpha}(u-w)\|+h\|\nabla \left(~^{C}D^{\alpha}_{t}(u-w)\right)\|\lesssim ~h^{2}.
\end{equation}
\end{thm}
\begin{proof}
This theorem is proved by following the arguments of the  proof of Lemma 3.1 and Lemma 3.2 from \cite{cannon1990priori}. 
\end{proof}

\section{Linearized fractional Crank-Nicolson-Galerkin FEM}

\noindent In this section,  we obtain the fully discrete formulation of the problem \eqref{22-5-1} by discretizing the domain in the space direction as well as in the time direction. In the space direction  we use a conforming FEM \cite{thomee2007galerkin}  and in the time direction  we use graded mesh $t_{n}=T\left(\frac{n}{N}\right)^{r},~n=0,1,\dots,N$, with the  non-uniform time steps  $\tau_{n}=t_{n}-t_{n-1}, (n\geq 1)$. This graded mesh satisfies the following properties
 \cite{chen2019error,kopteva2019error,ren2021sharp}
\begin{equation}\label{24-5-1}
    \tau_{n}\leq r~ T N^{-r}n^{r-1}\leq r~T^{1/r}N^{-1}t_{n}^{1-1/r},~n=1,2,\dots,N,
\end{equation} 
and 
\begin{equation}\label{24-5-2}
    t_{n}\leq 2^{r}~t_{n-1},~n=2,3,\dots,N,
\end{equation}
and there exists a positive  constant $\gamma$  independent of step sizes such that 
\begin{equation}\label{24-5-3}
 \tau_{n-1} <   \tau_{n} < \gamma ~\tau_{n-1},~n=2,\dots,N.
\end{equation}
Let us denote $u^{n}=u(t_{n})~\text{ and }~  u^{n-\sigma}= u(t_{n-\sigma}),$ where $t_{n-\sigma}=(1-\sigma)~t_{n}+\sigma ~t_{n-1}$, for $0\leq \sigma \leq 1$.\\

\noindent In Section \ref{8-7-2a}, we derive the local truncation errors of the approximations of the Caputo fractional derivative, the nonlinear diffusion coefficient, and  the memory term. Based on these approximations we devise a fully discrete numerical scheme \eqref{27-5-9a}-\eqref{27-5-9b} and derive a priori bounds on its solutions in Section \ref{8-7-3}. With an application of Br\"{o}uwer fixed point theorem,  we establish the well-posedness of the developed fully discrete numerical scheme in  Section \ref{8-7-4}. In Section \ref{8-7-5}, we derive the global rate of convergence of the fully discrete numerical scheme.

\subsection{Local truncation errors}\label{8-7-2a}
\noindent \textbf{ Approximation of Caputo fractional derivative: }
In this scheme \cite{alikhanov2015new}, Caputo fractional derivative is evaluated at $t_{n-\sigma}~(n\geq 1)$ as follows 
\begin{equation*}\label{SA1}
\begin{aligned}
    ^{C}\mathcal{D}^{\alpha}_{t_{n-\sigma}}u&=\frac{1}{\Gamma(1-\alpha)}\int_{0}^{t_{n-\sigma}}(t_{n-\sigma}-s)^{-\alpha}u'(s)~ds\\
    &=\frac{1}{\Gamma(1-\alpha)}\sum_{j=1}^{n-1}\int_{t_{j-1}}^{t_{j}}(t_{n-\sigma}-s)^{-\alpha}u'(s)~ds\\
    &+\frac{1}{\Gamma(1-\alpha)}\int_{t_{n-1}}^{t_{n-\sigma}}(t_{n-\sigma}-s)^{-\alpha}u'(s)~ds.
    \end{aligned}
\end{equation*}
For the approximation of Caputo fractional derivative at $t_{n-\sigma}$, we use the quadratic interpolation $\mathcal{P}_{2,j}u(t)$ of $u(t)$ in the interval $[t_{j-1},t_{j}]~(1\leq j \leq n-1)$ and the linear interpolation $\mathcal{P}_{1,n}u(t)$ of $u(t)$ in the interval $[t_{n-1},t_{n}].$\\

\noindent Let $~^{\sigma}\mathbb{D}^{\alpha}_{t}u^{n}$ be the L2-1$_\sigma$ approximation of the Caputo fractional derivative at $t_{n-\sigma}$ and $\mathcal{T}_{u}^{n-\sigma}$ be the local truncation error associated with $u$ at $t_{n-\sigma}$. Then $~^{\sigma}\mathbb{D}^{\alpha}_{t}u^{n}~(1\leq n \leq N)$ is given by 
\begin{equation}\label{SA6}
\begin{aligned}
    ~^{\sigma}\mathbb{D}^{\alpha}_{t}u^{n}=
    c_{n,n}u^{n}+\sum_{j=1}^{n-1}\left(c_{n,j}-c_{n,j+1}\right)u^{j}-c_{n,1}u^{0},
    \end{aligned}
\end{equation}
where $c_{1,1}=\tau_{1}^{-1}a_{1,1}$ for $n=1$ and for $n\geq 2$
\begin{equation*}
    c_{n,j}=\begin{cases}\tau_{j}^{-1}\left(a_{n,j}-b_{n,j}\right),& \text{for}~~ j=1\\
    \tau_{j}^{-1}\left(a_{n,j}-b_{n,j}+b_{n,j-1}\right),& \text{for}~~ 2\leq j \leq n-1\\
    \tau_{j}^{-1}\left(a_{n,j}+b_{n,j-1}\right), & \text{for} ~~j=n.
    \end{cases}
\end{equation*}
with

\begin{equation*}
    a_{n,j}=\frac{1}{\Gamma(1-\alpha)}\int_{t_{j-1}}^{t_{j}}(t_{n-\sigma}-s)^{-\alpha}~ds,~\text{for}~j=1,2,\dots,n-1,
\end{equation*}
and 
\begin{equation*}
    a_{n,n}=\frac{1}{\Gamma(1-\alpha)}\int_{t_{n-1}}^{t_{n-\sigma}}(t_{n-\sigma }-s)^{-\alpha}~ds=\frac{(1-\sigma)^{1-\alpha}}{\Gamma(2-\alpha)}\tau_{n}^{1-\alpha},
    \end{equation*}
    and
    \begin{equation*}
b_{n,j}=\frac{2}{t_{j+1}-t_{j-1}}\frac{1}{\Gamma(1-\alpha)}\int_{t_{j-1}}^{t_{j}}(t_{n-\sigma}-s)^{-\alpha}\left(s-t_{j-\frac{1}{2}}\right)~ds,~\text{for}~j=1,2,\dots,n-1.
\end{equation*}

\noindent The local truncation error $\mathcal{T}_{u}^{n-\sigma}$ is given by  $\mathcal{T}^{n-\sigma}_{u}=~^{C}\mathcal{D}^{\alpha}_{t_{n-\sigma}}u-~^{\sigma}\mathbb{D}^{\alpha}_{t}u^{n}$.\\

\noindent Using the idea of the proof of the Lemma 3 \cite{chen2019error}, one can  prove the following properties of the  weights $c_{n,j}~(1\leq j \leq n)$
\begin{equation}\label{8-7-1a}
    c_{n,j+1}> c_{n,j}>0~\text{for}~j=1,2,\dots,n-1,
\end{equation}
and   
\begin{equation}\label{8-7-1}
     c_{n,n}\leq \tau_{n}^{-\alpha}\left[\frac{(1-\sigma)^{1-\alpha}}{\Gamma(2-\alpha)}+\frac{\alpha}{6\Gamma(1-\alpha)}\frac{1}{(1-\sigma)^{1+\alpha}}\right],
\end{equation}
and
\begin{equation}\label{8-7-2}
    c_{n,n} \geq \tau_{n}^{-\alpha} \left[\frac{(1-\sigma)^{1-\alpha}}{\Gamma(2-\alpha)}+\frac{\alpha}{6\Gamma(1-\alpha)}\frac{1}{(1+\gamma)}\frac{1}{((1-\sigma)+\gamma)^{1+\alpha}}\right].
\end{equation}

\noindent The approximation  $~^{\sigma}\mathbb{D}^{\alpha}_{t}u^{n}$ of Caputo fractional derivative satisfies the following positivity property (\cite{chen2019error},~Lemma 8) 
\begin{equation}\label{22-5-1a}
    \left(~^{\sigma}\mathbb{D}^{\alpha}_{t}u^{n},u^{n}\right)\geq \frac{1}{2}~^{\sigma}\mathbb{D}^{\alpha}_{t}\|u^{n}\|^{2}. 
\end{equation}
\begin{lem} \label{25-5-1} (\cite{chen2019error}, Lemma 7) Suppose that $u \in C[0,T]\cap C^{3}(0,T]$  and satisfies the regularity assumption \eqref{28-6-1}. Then the local truncation error $\mathcal{T}_{u}^{n-\sigma}$ satisfies 
\begin{equation}\label{25-5-2}
    \left|\mathcal{T}_{u}^{n-\sigma}\right|\lesssim t_{n-\sigma}^{-\alpha}~N^{-\min\left\{r\alpha,~ 3-\alpha\right\}},~\text{for}~n=1,2,\dots,N.
\end{equation}
\end{lem}

\noindent \textbf{Approximation of nonlinear diffusion coefficient: } Due to the presence of nonlocal nonlinear diffusion coefficient $(1+\|\nabla u\|^{2})$ we obtain a nonlinear system of algebraic equations with non sparse Jacobian \cite{gudi2012finite}. To solve this nonlinear system the Newton-Raphson method is applied, but for that  we need huge amount of computational cost and computer storage. We reduce these costs by developing the following linearized approximations  of diffusion coefficient at $t_{n-\sigma}$
\begin{equation}\label{25-5-3}
    \tilde{u}^{n-1,\sigma}=u^{n-1}+(1-\sigma)\left(\frac{\tau_{n}}{\tau_{n-1}}\right)\left(u^{n-1}-u^{n-2}\right),~\text{for}~n\geq 2,
\end{equation}
and 
\begin{equation}\label{25-5-4}
    u^{n,\sigma}=(1-\sigma)~u^{n}+\sigma~u^{n-1},~\text{for}~n\geq 1.
\end{equation}
\begin{rmk}
If we take uniform mesh $(\tau_{n}=\tau_{n-1})$ and $\sigma =\frac{\alpha}{2}$ in \eqref{25-5-3} and \eqref{25-5-4} then we get linearized fractional Crank-Nicolson scheme which we have developed in \cite{lalit}.
\end{rmk}

\begin{lem}\label{25-5-5} Suppose that $u\in C[0,T]\cap C^{2}(0,T]$ which satisfies regularity assumption \eqref{28-6-1}. Then linearization errors $\tilde{\mathcal{L}}_{u}^{n-\sigma}=\tilde{u}^{n-1,\sigma}-u^{n-\sigma}$ and $\mathcal{L}_{u}^{n-\sigma}=u^{n,\sigma}-u^{n-\sigma}$ satisfy
\begin{equation}\label{25-5-6}
    \left|\tilde{\mathcal{L}}_{u}^{n-\sigma}\right|\lesssim  N^{-\min\{r\alpha,2\}},~~ \text{for}~n\geq 2,
\end{equation}
and 
\begin{equation}\label{25-5-7}
    \left|\mathcal{L}_{u}^{n-\sigma}\right| \lesssim  N^{-\min\{r\alpha,2\}},~~ \text{for}~n\geq 1.
\end{equation}
\end{lem}

\begin{proof} By the  Taylor series expansion of $u^{n-1}$ and $u^{n-2}$ around $ u^{n-\sigma}$ there exist  points $\xi_{1}\in (t_{n-1},t_{n-\sigma})$ and $\xi_{2}\in (t_{n-2},t_{n-\sigma})$ such that 
\begin{equation*}\label{27-5-1}
    \tilde{\mathcal{L}}_{u}^{n-\sigma}=\frac{1}{2}(t_{n-1}-t_{n-\sigma})(t_{n-2}-t_{n-\sigma})^{2}u_{tt}(\xi_{1})-\frac{1}{2}(t_{n-2}-t_{n-\sigma})(t_{n-1}-t_{n-\sigma})^{2}u_{tt}(\xi_{2}).
\end{equation*}
Using intermediate value theorem (\cite{baskar2016introduction}, Theorem 6.14) for $u_{tt}$ we obtain a point $\xi \in (t_{n-2},t_{n-\sigma})$ such that 
\begin{equation*}\label{27-5-2}
    \tilde{\mathcal{L}}_{u}^{n-\sigma}=\left[1+(1-\sigma)\frac{\tau_{n}}{\tau_{n-1}}\right](1-\sigma)\tau_{n}\tau_{n-1}u_{tt}(\xi).
\end{equation*}
In the view of estimates \eqref{24-5-3} and $|u_{tt}(\xi)|\lesssim \xi^{\alpha-2}\lesssim t_{n-2}^{\alpha-2}$ we get 
\begin{equation*}\label{27-5-3}
    \left|\tilde{\mathcal{L}}_{u}^{n-\sigma}\right|\lesssim \tau_{n}^{2}~t_{n-2}^{\alpha-2}.
\end{equation*}
Using \eqref{24-5-2} and \eqref{24-5-1} we obtain 
\begin{equation*}\label{27-5-4}
    \left|\tilde{\mathcal{L}}_{u}^{n-\sigma}\right|\lesssim n^{r\alpha-2}N^{-r\alpha} \lesssim N^{-\min\{r\alpha,2\}}.
\end{equation*}
Similarly we can prove the estimate \eqref{25-5-7}.
\end{proof}

\noindent \textbf{Approximation of memory term:} 
For $n=1$, we apply right rectangle  rule on the interval $[0,t_{1-\sigma}]$ and define the quadrature error operator $\left(\mathcal{Q}^{1-\sigma}_{u},v\right)$ as follows
\begin{equation}\label{27-5-5}
    \left(\mathcal{Q}_{u}^{1-\sigma},v\right)=t_{1-\sigma}(\nabla u^{1-\sigma}, \nabla v)-\int_{0}^{t_{1-\sigma}}(\nabla u(s),\nabla v)~ds~\forall~v\in H^{1}_{0}(\Omega).
\end{equation}
For $n\geq 2$, we apply the modified trapezoidal rule in which we divide the interval $[t_{1},t_{n-\sigma}]$ into two parts $[t_{1},t_{n-1}]$ and $[t_{n-1},t_{n-\sigma}]$. In the first part we apply composite trapezoidal  rule and in the second part left rectangle rule is used. Quadrature error operator $\left(\mathcal{Q}_{u}^{n-\sigma},v\right),~n\geq 2$ is defined as follows
\begin{equation}\label{27-5-6}
\begin{aligned}
     \left(\mathcal{Q}^{n-\sigma}_{u},v\right)&=\sum_{j=2}^{n-1}\frac{\tau_{j}}{2}\left(\nabla u^{j-1}+\nabla u^{j},\nabla v\right)-\int_{t_{1}}^{t_{n-1}}(\nabla u(s),v)~ds\\
     &+(1-\sigma)\tau_{n}(\nabla u^{n-1},\nabla v)-\int_{t_{n-1}}^{t_{n-\sigma}}(\nabla u(s),\nabla v)~ds~\forall~v\in H^{1}_{0}(\Omega).
     \end{aligned}
\end{equation}
\begin{lem}  \label{27-5-7} Suppose that $u\in C[0,T]\cap C^{2}(0,T]$ which satisfies the regularity assumption \eqref{28-6-1}. Then the quadrature errors $\mathcal{Q}_{u}^{1-\sigma}$ and $\mathcal{Q}_{u}^{n-\sigma}~(n\geq 2)$ satisfy 
\begin{align}
    \left(\mathcal{Q}_{u}^{1-\sigma},v\right)&\lesssim N^{-r(1+\alpha)}\|\nabla v\| ~\forall~ v~\in~H^{1}_{0}(\Omega),\label{27-5-8}\\
    \left(\mathcal{Q}_{u}^{n-\sigma},v\right)&\lesssim N^{-\min\{r(1+\alpha),2\}}\|\nabla v\|~\forall~ v~\in~H^{1}_{0}(\Omega),\label{27-5-9}
\end{align}

\end{lem}
\begin{proof}
Let us denote $\bar{u}(s)=(\nabla u(s),\nabla v)$ then $\left(\mathcal{Q}^{1-\sigma}_{u},v\right)$ is rewritten as 
\begin{equation*}
    \left(\mathcal{Q}_{u}^{1-\sigma},v\right)=t_{1-\sigma}\bar{u}(t_{1-\sigma})-\int_{0}^{t_{1-\sigma}}\bar{ u}(s)~ds.
\end{equation*}
 For $s \in (0,t_{1-\sigma})$, the function $\bar{u}(s)$ is continuous on $[s,t_{1-\sigma}]$ and differentiable on $(s,t_{1-\sigma})$ then there exists a point $\xi_{1} \in (s,t_{1-\sigma}$) (thanks to the mean value theorem) such that 
\begin{equation*}\label{9-7-1}
\begin{aligned}
    \bar{u}(t_{1-\sigma})-\bar{u}(s)=(t_{1-\sigma}-s)\bar{u}'({\xi_{1}}).
    \end{aligned}
\end{equation*}
Apply regularity assumption \eqref{28-6-1} and Cauchy-Schwarz inequality  to get
\begin{equation*}
\begin{aligned}
    \bar{u}(t_{1-\sigma})-\bar{u}(s)\lesssim (t_{1}-s)\xi_{1}^{\alpha-1}\|\nabla v\|\lesssim (t_{1}-s)s^{\alpha-1}\|\nabla v\|.
    \end{aligned}
\end{equation*}
Integrating both sides over $(t_{0},t_{1-\sigma})$ we obtain 
\begin{equation*}
\begin{aligned}
    \left(\mathcal{Q}_{u}^{1-\sigma},v\right)&\lesssim\|\nabla v\| \int_{t_{0}}^{t_{1}}(t_{1}-s)s^{\alpha-1}ds \lesssim \|\nabla v\|t_{1}^{1+\alpha}\lesssim N^{-r(1+\alpha)}\|\nabla v\|.
    \end{aligned}
\end{equation*}
 Now we estimate $\left(\mathcal{Q}_{u}^{n-\sigma},v\right)~(n\geq 2)$, consider
\begin{equation*}
        \begin{aligned}
    \left(\mathcal{Q}^{n-\sigma}_{u},v\right)&=\sum_{j=1}^{n-2}\int_{t_{j}}^{t_{j+1}}\left(s-t_{j}\right)(s-t_{j+1})\frac{\partial^{2}\bar{u}}{\partial s^{2}}(\eta_{j}(s))~ds\\
    &+\int_{t_{n-1}}^{t_{n-\sigma}}\left(s-t_{n-1}\right)\frac{\partial \bar{u}}{\partial s}(\xi(s))~ds\\
    &~\text{for some}~ \eta_{j}(s)\in \left(t_{j},t_{j+1}\right), \xi(s) \in \left(t_{n-1},t_{n-\sigma}\right).
    \end{aligned}
    \end{equation*}
    By mean value theorem for integral (\cite{burden2015numerical},~Theorem 1.13), there exists $\eta_{j}\in (t_{j},t_{j+1})$ and $\xi \in (t_{n-1},t_{n-\sigma})$ such that 
    \begin{equation}\label{S12}
        \begin{aligned}
    \left(\mathcal{Q}_{u}^{n-\sigma},v\right)&\leq \sum_{j=1}^{n-2}\max_{\eta_{j}\in \left(t_{j},t_{j+1}\right)}\left|\frac{\partial^{2}\bar{u}}{\partial s^{2}}(\eta_{j})\right|\frac{\tau_{j+1}^{3}}{6}+\max_{\xi\in \left(t_{n-1},t_{n-\sigma}\right)}\left|\frac{\partial \bar{u}}{\partial s}(\xi)\right|\frac{\tau_{n}^{2}}{2},
    \end{aligned}
    \end{equation}
 Using regularity assumption \eqref{28-6-1} and Cauchy-Schwarz inequality  we obtain
    \begin{equation}\label{S15}
        \begin{aligned}
    \left(\mathcal{Q}_{u}^{n-\sigma},v\right)
    &\lesssim \sum_{j=1}^{n-2}\tau_{j+1}^{3}t_{j}^{\alpha-2}\|\nabla v\|+\tau_{n}^{2}t_{n-1}^{\alpha-1}\|\nabla v\|.
    \end{aligned}
    \end{equation}
    Definition of $t_{j}$ and estimate \eqref{24-5-1} imply
    \begin{equation*}
    \begin{aligned}
        \tau_{j+1}^{3}t_{j}^{\alpha-2}\lesssim  j^{r(\alpha+1)-3}N^{-r(\alpha+1)}.
        \end{aligned}
    \end{equation*}
    Now using the well-known convergence results \cite{stynes2017error} for the series of the type  we have 
\begin{equation}\label{1a}
    \sum_{j=1}^{\left[\frac{n}{2}\right]-1}j^{r(\alpha+1)-3}n^{-r(\alpha+1)}\leq \begin{cases} n^{-r(\alpha+1)}&~ \text{if}~r(\alpha+1)<2,\\
    n^{-2}\ln n &~\text{if}~r(\alpha+1)=2,\\
    n^{-2}&~\text{if}~r(\alpha+1)>2.
    \end{cases}
\end{equation} 
Further $\left[\frac{n}{2}\right]\leq j \leq n-2$, we have 
\begin{equation*}
    \left|\sum_{j=\left[\frac{n}{2}\right]}^{n-2}\int_{t_{j}}^{t_{j+1}}\left(s-t_{j}\right)\left(s-t_{j+1}\right)\frac{\partial^{2}\bar{u}}{\partial s^{2}}(\eta_{j})ds\right|\leq \|\nabla v\|\sum_{j=\left[\frac{n}{2}\right]}^{n-2}\tau_{j+1}^{2}t_{j}^{\alpha-2}\int_{t_{j}}^{t_{j+1}}1~ds.
\end{equation*}
Using estimate \eqref{24-5-1} and $t_{j}^{\alpha-2}\lesssim t_{n}^{\alpha-2}$ for $\left[\frac{n}{2}\right]\leq j \leq n-2$, we get 
\begin{equation*}
\begin{aligned}
    \left|\sum_{j=\left[\frac{n}{2}\right]}^{n-2}\int_{t_{j}}^{t_{j+1}}\left(s-t_{j}\right)\left(s-t_{j+1}\right)\frac{\partial^{2}\bar{u}}{\partial s^{2}}(\eta_{j})ds\right|
    &\lesssim T^{2}N^{-2r}n^{2(r-1)}t_{n}^{\alpha-2}\int_{t_{\left[\frac{n}{2}\right]}}^{t_{n-1}}1~ds\|\nabla v\|\\
    &\lesssim T^{2}N^{-2r}n^{2(r-1)}t_{n}^{\alpha-2}t_{n-1}\|\nabla v\|\\
    &\lesssim T^{2}N^{-2r}n^{2(r-1)}t_{n}^{\alpha-1}\|\nabla v\|.
    \end{aligned}
\end{equation*}
 Put the value of $t_{n}$, we get 
\begin{equation}\label{1c}
\begin{aligned}
    \left|\sum_{j=\left[\frac{n}{2}\right]}^{n-2}\int_{t_{j}}^{t_{j+1}}\left(s-t_{j}\right)\left(s-t_{j+1}\right)\frac{\partial^{2}\bar{u}}{\partial s^{2}}(\eta_{j})ds\right|
    &\lesssim N^{-r(\alpha+1)}n^{r(1+\alpha)-2}\|\nabla v\|.
    \end{aligned}
\end{equation}
Similarly the last part of \eqref{S15} is estimated as follows 
\begin{equation}\label{1d}
    \tau_{n}^{2}t_{n-1}^{\alpha-1}\lesssim \tau_{n}^{2}t_{n}^{\alpha-1}\lesssim  N^{-r(\alpha+1)}n^{r(1+\alpha)-2}.
\end{equation}
Combine the estimates \eqref{1a}-\eqref{1d} in \eqref{S15} to deduce the result. 
\end{proof}

\subsection{Fully discrete formulation and a priori bounds on its solution}\label{8-7-3}
Based on the approximations discussed above in Section \ref{8-7-2a}, we develop the following linearized  fractional Crank-Nicolson-Galerkin fully discrete numerical scheme:  Find $u_{h}^{n}~(n\geq 1)$ such that following equations hold for all $v_{h}\in \mathcal{V}_{h}$,\\

\noindent For $n=1$
\begin{equation}\label{27-5-9a}
    \left(~^{\sigma}\mathbb{D}^{\alpha}_{t}u_{h}^{1},v_{h}\right)+\left(1+\|\nabla u_{h}^{1,\sigma}\|^{2}\right)\left(\nabla u_{h}^{1,\sigma},\nabla v_{h}\right)=\left(f_{h}^{1-\sigma},v_{h}\right)+(1-\sigma)\tau_{1}\left(\nabla u_{h}^{1,\sigma},\nabla v_{h}\right),
\end{equation}

\noindent for $n\geq 2$
\begin{equation}\label{27-5-9b}
\begin{aligned}
    \left(~^{\sigma}\mathbb{D}^{\alpha}_{t}u_{h}^{n},v_{h}\right)+\left(1+\|\nabla \tilde{u}_{h}^{n-1,\sigma}\|^{2}\right)\left(\nabla u_{h}^{n,\sigma},\nabla v_{h}\right)&=\left(f_{h}^{n-\sigma},v_{h}\right)
    +\sum_{j=1}^{n-1}\tilde{\tau}_{j}\left(\nabla u_{h}^{j},\nabla v_{h}\right),
    \end{aligned}
\end{equation}
where $\tilde{\tau}_{j}~(1\leq j \leq n-1)$ is defined by $\tilde{\tau}_{1}=(1-\sigma)\tau_{2}$ for $n=2$ and for $n \geq 3$

\begin{equation*}
   \tilde{\tau}_{j}=\begin{cases}
   \frac{\tau_{j+1}}{2}, & j=1\\
            \frac{\tau_{j}+\tau_{j+1}}{2}, & 2\leq j \leq n-2\\
            \frac{\tau_{j}}{2}+(1-\sigma)\tau_{j+1},& j=n-1.
\end{cases}
\end{equation*}
with initial condition $u_{h}^{0}=R_{h}u_{0}$. Using the property \eqref{24-5-3} of the graded mesh we have $\tilde{\tau}_{j}\lesssim \tau_{j},(1\leq j \leq n-1)$.\\

\noindent Now we derive a priori bounds on the solution of numerical scheme \eqref{27-5-9a}-\eqref{27-5-9b}. The presence of the memory term in \eqref{22-5-1} restricts the use of fractional Gronwall's inequality \cite{liao2019discrete} while  deriving the a priori bounds. We adress this issue by defining the  following \textbf{weighted $H^{1}(\Omega)$ norm }
\begin{equation}\label{22-6-4}
    \max_{0 \leq n \leq N}\||u_{h}^{n}\||=\||u_{h}^{m}\||,~ \text{for some}~m \in \{0,1,2,\dots,N\},
\end{equation}
where  $\||u_{h}^{n}\||= \|u_{h}^{n}\|+c_{n,n}^{-1/2}\|\nabla u_{h}^{n}\|, (n\geq 1)$ and $\||u_{h}^{0}\||=\|u_{h}^{0}\|$.

\begin{thm} The solutions of the fully discrete numerical scheme \eqref{27-5-9a}-\eqref{27-5-9b} satisfy the following a  priori bounds
\begin{equation}\label{1-7-1}
     \max_{0\leq n \leq N}\||u_{h}^{n}\||\lesssim  \|u_{h}^{0}\|+\max_{1\leq n \leq N}\|f_{h}^{n-\sigma}\|,
\end{equation}
and 
\begin{equation}\label{22-6-4y}
     \max_{0\leq n \leq N}\|\nabla u_{h}^{n}\|\lesssim  \|\nabla u_{h}^{0}\|+\max_{1\leq n \leq N}\|f_{h}^{n-\sigma}\|.
\end{equation}
\end{thm}
\begin{proof}
For $n\geq 2$, put $v_{h}=u_{h}^{n}$ in \eqref{27-5-9b} we get 
\begin{equation}\label{22-6-1}
\begin{aligned}
    \left(~^{\sigma}\mathbb{D}^{\alpha}_{t}u_{h}^{n},u_{h}^{n}\right)+\left(1+\|\nabla \tilde{u}_{h}^{n-1,\sigma}\|^{2}\right)\left(\nabla u_{h}^{n,\sigma},\nabla u_{h}^{n}\right)&=\left(f_{h}^{n-\sigma},u_{h}^{n}\right)
    +\sum_{j=1}^{n-1}\tilde{\tau}_{j}\left(\nabla u_{h}^{j},\nabla u_{h}^{n}\right).
    \end{aligned}
\end{equation}
Using \eqref{22-5-1a} and Cauchy-Schwarz inequality in \eqref{22-6-1} we have 
\begin{equation*}\label{22-6-2}
\begin{aligned}
    ~^{\sigma}\mathbb{D}^{\alpha}_{t}\|u_{h}^{n}\|^{2}+\|\nabla u_{h}^{n}\|^{2}\lesssim \|\nabla u_{h}^{n-1}\|~\|\nabla u_{h}^{n}\|+\|f_{h}^{n-\sigma}\|~\|u_{h}^{n}\|
    +\sum_{j=1}^{n-1}\tilde{\tau}_{j}\|\nabla u_{h}^{j}\|~\|\nabla u_{h}^{n}\|.
    \end{aligned}
\end{equation*}
Definition of $~^{\sigma}\mathbb{D}^{\alpha}_{t}\|u_{h}^{n}\|^{2}$ \eqref{SA6} yields 
\begin{equation}\label{22-6-3}
\begin{aligned}
     c_{n,n}\|u_{h}^{n}\|^{2}+\|\nabla u_{h}^{n}\|^{2}&\lesssim  \sum_{j=1}^{n-1}\left(c_{n,j+1}-c_{n,j}\right)\|u_{h}^{j}\|^{2}+c_{n,1}\|u_{h}^{0}\|^{2}+ \|f_{h}^{n-\sigma}\|~\|u_{h}^{n}\|\\
     &+\sum_{j=1}^{n-1}\tilde{\tau}_{j}^{'}\|\nabla u_{h}^{j}\|~\|\nabla u_{h}^{n}\|,
    \end{aligned}
\end{equation}
where $\tilde{\tau}_{j}^{'}$ is defined as 
\begin{equation}\label{22-6-3a}
   \tilde{\tau}_{j}^{'}=\begin{cases}
   \tilde{\tau}_{j}, & 1\leq j \leq n-2\\
            1+\tilde{\tau}_{j}, & j= n-1.
\end{cases}
\end{equation}
Using weighted $H^{1}(\Omega)$ norm  \eqref{22-6-4} in \eqref{22-6-3} we obtain
\begin{equation}\label{22-6-5}
\begin{aligned}
     \||u_{h}^{m}\||^{2}&\lesssim  c_{m,m}^{-1}\sum_{j=1}^{m-1}\left(c_{m,j+1}-c_{m,j}\right)\||u_{h}^{j}\||~\||u_{h}^{m}\||+c_{m,m}^{-1}c_{m,1}\|u_{h}^{0}\|~\||u_{h}^{m}\||\\
     &+c_{m,m}^{-1} \|f_{h}^{m-\sigma}\|~\||u_{h}^{m}\||
    +c_{m,m}^{-1/2}\sum_{j=1}^{m-1}\tilde{\tau}_{j}^{'}c_{j,j}^{1/2}\|| u_{h}^{j}\||~\|| u_{h}^{m}\||,
    \end{aligned}
\end{equation}
On cancelling $\||u_{h}^{m}\||$ from both sides of \eqref{22-6-5} we get 
\begin{equation}\label{22-6-6}
\begin{aligned}
     \||u_{h}^{m}\||&\lesssim  c_{m,m}^{-1}\sum_{j=1}^{m-1}\left(c_{m,j+1}-c_{m,j}\right)\||u_{h}^{j}\||+c_{m,m}^{-1}c_{m,1}\|u_{h}^{0}\|+c_{m,m}^{-1} \|f_{h}^{m-\sigma}\|\\
     &
    +c_{m,m}^{-1/2}\sum_{j=1}^{m-1}\tilde{\tau}_{j}^{'}~c_{j,j}^{1/2}\|| u_{h}^{j}\||.
    \end{aligned}
\end{equation}
 We  apply the  discrete Gronwall inequality to obtain 
\begin{equation}\label{22-6-7}
\begin{aligned}
     \||u_{h}^{m}\||&\lesssim  \exp\left(\sum_{j=1}^{m-1}c_{m,m}^{-1}\left(c_{m,j+1}-c_{m,j}\right)+c_{m,m}^{-1/2}\tilde{\tau}_{j}^{'}~c_{j,j}^{1/2}\right)\left(c_{m,m}^{-1}c_{m,1}\|u_{h}^{0}\|
     +c_{m,m}^{-1} \|f_{h}^{m-\sigma}\|\right).
    \end{aligned}
\end{equation}
Let us denote $w_{m,j}=c_{m,m}^{-1}\left(c_{m,j+1}-c_{m,j}\right)+c_{m,m}^{-1/2}\tilde{\tau}_{j}^{'}~c_{j,j}^{1/2}$ then 
\begin{equation}\label{22-6-7b}
    \sum_{j=1}^{m-1}w_{m,j}=1-c_{m,m}^{-1}c_{m,1}+\sum_{j=1}^{m-1}c_{m,m}^{-1/2}\tilde{\tau}_{j}^{'}~c_{j,j}^{1/2}, 
\end{equation}
Using \eqref{8-7-1a} and  the definition \eqref{22-6-3a} of $\tilde{\tau}_{j}^{'}$ we have 
\begin{equation*}\label{22-6-7c}
    \sum_{j=1}^{m-1}w_{m,j}\lesssim 1+c_{m,m}^{-1/2}c_{m-1,m-1}^{1/2}+\sum_{j=1}^{m-1}c_{m,m}^{-1/2}\tilde{\tau}_{j}~c_{j,j}^{1/2}. 
\end{equation*}
Using the properties of weights $c_{m,j}$ \eqref{8-7-1}-\eqref{8-7-2} we get 
\begin{equation*}\label{22-6-7d}
    \sum_{j=1}^{m-1}w_{m,j}\lesssim 1+\tau_{m}^{\alpha/2}\tau_{m-1}^{-\alpha/2}+\tau_{m}^{\alpha/2}\sum_{j=1}^{m-1}\tilde{\tau}_{j}~\tau_{j}^{-\alpha/2}, 
\end{equation*}
apply the graded mesh property  \eqref{24-5-3} and $\tilde{\tau}_{j}\lesssim \tau_{j}$ to obtain 
\begin{equation}\label{22-6-7e}
\begin{aligned}
    \sum_{j=1}^{m-1}w_{m,j}&\lesssim 1+\gamma^{\alpha/2} +\tau_{m}^{\alpha/2}\sum_{j=1}^{m-1}\tau_{j}^{1-\alpha/2}\\
    &\lesssim 1+\gamma^{\alpha/2} +\tau_{m}^{\alpha/2}\sum_{j=1}^{m}\tau_{j}^{1-\alpha/2}.
    \end{aligned}
\end{equation}
Using graded mesh property \eqref{24-5-1} we have 
\begin{equation}\label{6-8-1}
\begin{aligned}
    \sum_{j=1}^{m-1}w_{m,j}
    &\lesssim 1+\gamma^{\alpha/2} +\sum_{j=1}^{m}N^{-\alpha/2}t_{m}^{(1-1/r)(\alpha/2)}N^{-(1-\alpha/2)}t_{j}^{(1-1/r)(1-\alpha/2)}\\
    &\lesssim 1+\gamma^{\alpha/2}+ N^{-1}\sum_{j=1}^{m}T\\
    & \lesssim 1+\gamma^{\alpha/2}+ N^{-1}NT=  1+\gamma^{\alpha/2}+ T.
    \end{aligned}
\end{equation}
Use \eqref{6-8-1}  in \eqref{22-6-7} to deduce 
\begin{equation*}
     \||u_{h}^{m}\||\lesssim c_{m,m}^{-1}c_{m,1}\|u_{h}^{0}\|
     +c_{m,m}^{-1} \|f_{h}^{m-\sigma}\|.
\end{equation*}
The properties \eqref{8-7-1a}-\eqref{8-7-2} of graded mesh and $\tau_{m}^{\alpha} < t_{m}^{\alpha} \leq T^{\alpha}$ imply 
\begin{equation*}
     \||u_{h}^{m}\||\lesssim \|u_{h}^{0}\|
     +\|f_{h}^{m-\sigma}\|.
\end{equation*}
Using the similar technique, we can get  a priori bound for the case $n=1$ also. Now we prove the estimate \eqref{22-6-4y}. By the definition of discrete Laplacian operator \eqref{8-7-1w}, the equation \eqref{27-5-9b} can be rewritten as 
\begin{equation}\label{27-5-9bb}
\begin{aligned}
    \left(~^{\sigma}\mathbb{D}^{\alpha}_{t}u_{h}^{n},v_{h}\right)+\left(1+\|\nabla \tilde{u}_{h}^{n-1,\sigma}\|^{2}\right)\left(-\Delta_{h} u_{h}^{n,\sigma}, v_{h}\right)&=\left(f_{h}^{n-\sigma},v_{h}\right)
    +\sum_{j=1}^{n-1}\tilde{\tau}_{j}\left(-\Delta_{h} u_{h}^{j}, v_{h}\right).
    \end{aligned}
\end{equation}
Put $v_{h}=-\Delta_{h}u_{h}^{n}$ in \eqref{27-5-9bb} and follow the same steps as we obtain estimate \eqref{1-7-1}.
\end{proof}

\subsection{ Well-posedness of the fully discrete formuation \eqref{27-5-9a}-\eqref{27-5-9b} }\label{8-7-4}

\begin{thm}\label{1-7-2} \textbf{(Existence)}
There exists a solution to the fully discrete numerical scheme \eqref{27-5-9a}-\eqref{27-5-9b}.
\end{thm}
\begin{proof}
For $n\geq 2$, the equation \eqref{27-5-9b} is linear in $u_{h}^{n}$ and the corresponding coefficient matrix is positive definite which ensures the existence and uniqueness of solution $u_{h}^{n}~(n\geq 2)$ of the problem \eqref{27-5-9b}. To show the existence of the solution $u_{h}^{1}$, we use the following variant of Br\"{o}uwer fixed point theorem \cite{thomee2007galerkin} which says that if  $\mathbb{H}$ be a    finite dimensional Hilbert space and  $G:\mathbb{H}\rightarrow \mathbb{H}$ be a continuous map such that $\left(G(\hat{w}),\hat{w}\right)>0$  for all $\hat{w}$ in $\mathbb{H}$ with $\|\hat{w}\|=r_{0},~r_{0}>0$.  Then there exists a $\tilde{w}$ in $\mathbb{H}$ such that $G(\tilde{w})=0$ and $\|\tilde{w}\|\leq r_{0}.$ \\

\noindent Put the definition of $^{\sigma}\mathbb{D}^{\alpha}_{t}u_{h}^{1}$ in \eqref{27-5-9a} we  get 
\begin{equation}\label{23-6-8}
    \left(c_{1,1}(u_{h}^{1}-u_{h}^{0}),v_{h}\right)+\left(1+\|\nabla u_{h}^{1,\sigma}\|^{2}\right)\left(\nabla u_{h}^{1,\sigma},\nabla v_{h}\right)=\left(f_{h}^{1-\sigma},v_{h}\right)+(1-\sigma)\tau_{1}\left(\nabla u_{h}^{1,\sigma},\nabla v_{h}\right).
\end{equation}
Multiply \eqref{23-6-8} by $(1-\sigma)$ to obtain 
\begin{equation}\label{23-6-9}
\begin{aligned}
    (u_{h}^{1,\sigma},v_{h})-(u_{h}^{0},v_{h})&+(1-\sigma)c_{1,1}^{-1}\left(1+\|\nabla u_{h}^{1,\sigma}\|^{2}\right)\left(\nabla u_{h}^{1,\sigma},\nabla v_{h}\right)\\
    &=c_{1,1}^{-1}(1-\sigma)\left(f_{h}^{1-\sigma},v_{h}\right)+c_{1,1}^{-1}(1-\sigma)^{2}\tau_{1}\left(\nabla u_{h}^{1,\sigma},\nabla v_{h}\right).
    \end{aligned}
\end{equation}
In the view of \eqref{23-6-9} we define a map $G:\mathcal{V}_{h}\rightarrow \mathcal{V}_{h}$ as 
\begin{equation}\label{23-6-10}
\begin{aligned}
    (G(u_{h}^{1,\sigma}),v_{h})&= (u_{h}^{1,\sigma},v_{h})-(u_{h}^{0},v_{h})+(1-\sigma)c_{1,1}^{-1}\left(1+\|\nabla u_{h}^{1,\sigma}\|^{2}\right)\left(\nabla u_{h}^{1,\sigma},\nabla v_{h}\right)\\
    &-c_{1,1}^{-1}(1-\sigma)\left(f_{h}^{1-\sigma},v_{h}\right)-c_{1,1}^{-1}(1-\sigma)^{2}\tau_{1}\left(\nabla u_{h}^{1,\sigma},\nabla v_{h}\right).
    \end{aligned}
\end{equation}
Put $v_{h}=u_{h}^{1,\sigma}$ in \eqref{23-6-10} then it is easy to show that $(G(u_{h}^{1,\sigma}),u_{h}^{1,\sigma})>0$ for sufficiently small $\tau_{1}$    and $G$ is continuous \cite{kumar2020finite,lalit}. Therefore above stated variant of Br\"{o}uwer fixed point theorem implies the existence of $u_{h}^{1,\sigma}$. Hence existence of $u_{h}^{1}$ follows.
\end{proof}
\begin{thm}\textbf{(Uniqueness)}
The solution of the fully discrete numerical scheme \eqref{27-5-9a}-\eqref{27-5-9b} is unique.
\end{thm}
\begin{proof}
Uniqueness of the solution $u_{h}^{n},(n \geq 2)$  is given by Theorem \ref{1-7-2}. For $n=1$, let $U_{h}^{1}$ and $V_{h}^{1}$ be the solutions of the numerical scheme $\eqref{27-5-9a}$ then $Z_{h}^{1}=U_{h}^{1}-V_{h}^{1}$ satisfies
\begin{equation}\label{2-7-1}
\begin{aligned}
     \left(~^{\sigma}\mathbb{D}^{\alpha}_{t}Z_{h}^{1},v_{h}\right)&+\left(\left(1+\|\nabla U_{h}^{1,\sigma}\|^{2}\right)\nabla U_{h}^{1,\sigma}-\left(1+\|\nabla V_{h}^{1,\sigma}\|^{2}\right)\nabla V_{h}^{1,\sigma},\nabla v_{h}\right)\\
     &=(1-\sigma)\tau_{1}\left(\nabla Z_{h}^{1,\sigma},\nabla v_{h}\right).
\end{aligned}
\end{equation}
Put $v_{h}=Z_{h}^{1,\sigma}$ in \eqref{2-7-1} and using $\left(~^{\sigma}\mathbb{D}^{\alpha}_{t}Z_{h}^{1},Z_{h}^{1,\sigma}\right)\geq \frac{1}{2}~^{\sigma}\mathbb{D}^{\alpha}_{t}\|Z_{h}^{1}\|^{2}$ (\cite{chen2019error},~Lemma 8) along with the monotonicity property of the Kirchhoff term ( \cite{kundu2016kirchhoff}, Lemma 2.8) we obtain 
\begin{equation}\label{2-7-2}
    \|Z_{h}^{1}\|^{2}+c_{1,1}^{-1}\left(1-(1-\sigma)\tau_{1}\right)\|\nabla Z_{h}^{1,\sigma}\|^{2}\leq 0.
\end{equation}
For sufficiently small $\tau_{1}$, we can deduce the uniqueness of the solution of \eqref{27-5-9a}.

\end{proof}

\subsection{Global  convergence analysis}\label{8-7-5}
In this section, we prove the main result of this article. 
\begin{thm}\label{9-7-2} Suppose that the solution $u$ of the problem \eqref{22-5-1} satisfies the realistic regularity assumption \eqref{28-6-1}. Then the solution $u_{h}^{n}$ of the fully discrete numerical scheme \eqref{27-5-9a}-\eqref{27-5-9b} converges to the solution $u$ with the following rate of accuracy 
\begin{equation}\label{8-7-1z}
    \max_{1\leq n \leq N}\|u^{n}-u_{h}^{n}\| \lesssim M^{-1}+N^{-\min\{r\alpha,~2\}},
\end{equation}
and 
\begin{equation}\label{8-7-1y}
     \max_{1\leq n \leq N}\|\nabla u^{n}-\nabla u_{h}^{n}\| \lesssim M^{-1}+N^{-\min\{r\alpha,~2\}}.
\end{equation}

\end{thm}
\begin{proof} 
Let us write $u^{n}-u_{h}^{n}=u^{n}-w^{n}+w^{n}-u_{h}^{n}$. Further denote $u^{n}-w^{n}=\rho^{n}$ and $u_{h}^{n}-w^{n}=\theta^{n}$. Using weak formulation \eqref{22-5-2}, modified Ritz-Volterra projection operator \eqref{23-6-4} in the fully discrete numerical scheme \eqref{27-5-9b} we get 
\begin{equation}\label{6-7-1}
\begin{aligned}
    &\left(^{\sigma}\mathbb{D}^{\alpha}_{t}\theta^{n},v_{h}\right)+\left(1+\|\nabla \tilde{u}_{h}^{n-1,\sigma}\|^{2}\right)\left(\nabla \theta^{n,\sigma},\nabla v_{h}\right)\\
    &=\left(^{C}\mathcal{D}^{\alpha}_{t_{n-\sigma}}\rho,v_{h}\right)+\left(\mathcal{T}_{w}^{n-\sigma},v_{h}\right)+\left(1+\|\nabla u^{n-\sigma}\|^{2}\right)\left(\nabla \mathcal{L}_{w}^{n-\sigma},\nabla v_{h}\right)\\
    &+\left(\|\nabla u^{n-\sigma}\|+\|\nabla \tilde{u}_{h}^{n-1,\sigma}\|\right)\left(\|\nabla \tilde{\mathcal{L}}_{u}^{n-\sigma}+\nabla \tilde{\rho}^{n-1,\sigma}+\nabla \tilde{\theta}^{n-1,\sigma}\|\right)\left(\nabla w^{n,\sigma},\nabla v_{h}\right)\\
    &+\left(\mathcal{Q}^{n-\sigma}_{w},\theta^{n}\right)+\sum_{j=1}^{n-1}\tilde{\tau}_{j}(\nabla \theta^{j},\nabla v_{h}).
    \end{aligned}
\end{equation}
Put $v_{h}=\theta^{n}$ in \eqref{6-7-1} and using   \eqref{22-5-1a} along with  Cauchy-Schwarz inequality  we obtain 
\begin{equation}\label{6-7-2}
\begin{aligned}
    &^{\sigma}\mathbb{D}^{\alpha}_{t}\|\theta^{n}\|^{2}+\|\nabla \theta^{n}\|^{2}\\
    &\lesssim \|^{C}\mathcal{D}^{\alpha}_{t_{n-\sigma}}\rho\|~\| \theta^{n}\|+\|\mathcal{T}_{w}^{n-\sigma}\|~\| \theta^{n}\|+\left(1+\|\nabla u^{n-\sigma}\|^{2}\right)\|\nabla \mathcal{L}_{w}^{n-\sigma}\|~\|\nabla \theta^{n}\|\\
    &+\left(\|\nabla u^{n-\sigma}\|+\|\nabla \tilde{u}_{h}^{n-1,\sigma}\|\right)\left(\|\nabla \tilde{\mathcal{L}}_{u} ^{n-\sigma}\|+\|\nabla \tilde{\rho}^{n-1,\sigma}\|+\|\nabla \tilde{\theta}^{n-1,\sigma}\|\right)\|\nabla w^{n,\sigma}\|~\|\nabla \theta^{n}\|\\
    &+\left(\mathcal{Q}^{n-\sigma}_{w},\theta^{n}\right)+\sum_{j=1}^{n-1}\tilde{\tau}_{j}\|\nabla \theta^{j}\|~\|\nabla \theta^{n}\|+\|\nabla \theta^{n-1}\|~\|\nabla \theta^{n}\|.
    \end{aligned}
\end{equation}
Definition of $~^{\sigma}\mathbb{D}^{\alpha}_{t}\|\theta^{n}\|^{2}$ \eqref{SA6}, approximation properties of modified Ritz-Volterra projection operator \eqref{23-6-5a}-\eqref{23-6-7} and local truncation errors \eqref{25-5-2}, \eqref{25-5-6}, \eqref{25-5-7}, \eqref{27-5-9}  yield
\begin{equation}\label{6-7-2a}
\begin{aligned}
    &\|\theta^{n}\|^{2}+c_{n,n}^{-1}\|\nabla \theta^{n}\|^{2}\\
    &\lesssim c_{n,n}^{-1}h^{2}~\| \theta^{n}\|+c_{n,n}^{-1}t_{n-\sigma}^{-\alpha}~N^{-\min\left\{r\alpha,~ 3-\alpha\right\}}~\| \theta^{n}\|\\
    &+c_{n,n}^{-1}\left(1+\|\nabla u^{n-\sigma}\|^{2}\right)N^{-\min\left\{r\alpha,~ 2\right\}}~\|\nabla \theta^{n}\|\\
    &+c_{n,n}^{-1}\left(\|\nabla u^{n-\sigma}\|+\|\nabla \tilde{u}_{h}^{n-1,\sigma}\|\right)\left(N^{-\min\left\{r\alpha,~ 2\right\}}+h+\|\nabla \tilde{\theta}^{n-1,\sigma}\|\right)\|\nabla w^{n,\sigma}\|\|\nabla \theta^{n}\|\\
    &+c_{n,n}^{-1}N^{-\min\left\{r(1+\alpha),~2 \right\}}\|\nabla \theta^{n}\|+c_{n,n}^{-1}\sum_{j=1}^{n-1}\tilde{\tau}_{j}\|\nabla \theta^{j}\|~\|\nabla \theta^{n}\|+c_{n,n}^{-1}\|\nabla \theta^{n-1}\|~\|\nabla \theta^{n}\|\\
    &+c_{n,n}^{-1}\sum_{j=1}^{n-1}\left(c_{n,j+1}-c_{n,j}\right)\|\theta^{j}\|^{2}.
    \end{aligned}
\end{equation}
Using the definition of weighted $H^{1}(\Omega)$ norm \eqref{22-6-4} and  a priori bound  on weak solution $u$ \eqref{23-6-2}, fully discrete solution $u_{h}^{n}$ \eqref{22-6-4y} and modified Ritz-Volterra projection operator \eqref{23-6-5aa} in \eqref{6-7-2a} we obtain 
\begin{equation}\label{6-7-2a1}
\begin{aligned}
     \||\theta^{m}\||
    &\lesssim c_{m,m}^{-1}h^{2}+c_{m,m}^{-1}t_{m-\sigma}^{-\alpha}~N^{-\min\left\{r\alpha,~ 3-\alpha\right\}}+c_{m,m}^{-1/2}N^{-\min\left\{r\alpha,~ 2\right\}}\\
    &+c_{m,m}^{-1/2}\left(N^{-\min\left\{r\alpha,~ 2\right\}}+h\right)+c_{m,m}^{-1/2}N^{-\min\left\{r(1+\alpha),~2 \right\}}\\
    &+c_{m,m}^{-1/2}\sum_{j=1}^{m-1}\tilde{w}_{j}c_{j,j}^{1/2}\|| \theta^{j}\|| +c_{m,m}^{-1}\sum_{j=1}^{m-1}\left(c_{m,j+1}-c_{m,j}\right)\||\theta^{j}\||,
    \end{aligned}
\end{equation}
where 
\begin{equation}
   \tilde{w}_{j}=\begin{cases}
   \tilde{\tau}_{j}, & 1\leq j \leq m-3,\\
            1+\tilde{\tau}_{j}, & j= m-2,\\
            2+\tilde{\tau}_{j}, & j= m-1.
\end{cases}
\end{equation}
Using  $c_{m,m}^{-1}\lesssim  \tau_{m}^{\alpha} \lesssim t_{m}^{\alpha}$ and discrete Gronwall inequality we obtain 
 \begin{equation}
 \begin{aligned}
     \||\theta^{m}\||&\lesssim \left(h+N^{-\min\left\{r\alpha,~ 2\right\}}\right)\exp\left(\sum_{j=1}^{m-1}\tilde{w}_{m,j}\right),
     \end{aligned}
 \end{equation}
 where $
     \tilde{w}_{m,j}=c_{m,m}^{-1/2}\tilde{w}_{j}c_{j,j}^{1/2}+c_{m,m}^{-1}\left(c_{m,j+1}-c_{m,j}\right)$.
Follow the proof of estimate \eqref{22-6-7e} to conclude 
 \begin{equation}
     \||\theta^{m}\||\lesssim \left(h+N^{-\min\left\{r\alpha,~ 2\right\}}\right).
 \end{equation}
  Finally triangle inequality and the best approximation property \eqref{23-6-5a} implies 
 \begin{equation}
     \|u^{m}-u_{h}^{m}\| \lesssim \|\rho^{m}\|+\|\theta^{m}\|\lesssim h^{2}+ h+N^{-\min\left\{r\alpha,~ 2\right\}} \lesssim h+N^{-\min\left\{r\alpha,~ 2\right\}}.
 \end{equation}
 By choosing $h=M^{-1}$, where $M$ is the degree of freedom in the space direction  we get
 \begin{equation}\label{8-7-1x}
     \|u^{m}-u_{h}^{m}\|\lesssim \left(M^{-1}+N^{-\min\left\{r\alpha,~ 2\right\}}\right).
 \end{equation}
 Similar approach will work for the case  $n=1$ also.\\
 
 \noindent To prove the estimate \eqref{8-7-1y} one case use \eqref{8-7-1z} but then we will have a loss of accuracy of order $\left(\frac{\alpha}{2}\right)$ as follows 
 \begin{equation}
     \|\nabla u^{m}-\nabla u_{h}^{m}\| \lesssim \tau_{m}^{-\alpha/2}\left(M^{-1}+N^{-\min\left\{r\alpha,~ 2\right\}}\right).
     \end{equation}
     To recover this loss of accuracy, we make use of the discrete Laplacian operator \eqref{8-7-1w}.
     Substitute $v_{h}=-\Delta_{h}\theta^{n}$ in \eqref{6-7-1} and follow the arguments of the proof of estimate \eqref{8-7-1x} to conclude \eqref{8-7-1y}.

\end{proof}

\section{Numerical experiment}\label{26-6-1}
In this section, we  present some  numerical results to  test the accuracy and efficiency of the proposed linearized fractional Crank-Nicolson-Galerkin scheme \eqref{27-5-9a}-\eqref{27-5-9b}.  At the first time level $t_{1}$, the scheme \eqref{27-5-9a} is nonlinear, so we apply  the modified  Newton-Raphson method \cite{gudi2012finite} with tolerance $10^{-7}$ as a stopping criterion. Using the value of $u_{h}^{1}$ from \eqref{27-5-9a}, we obtain $u_{h}^{n}~ (n\geq 2)$ \eqref{27-5-9b}.\\

\noindent The discrete $L^{\infty}(0,T;L^{2}(\Omega))$ and $L^{\infty}(0,T;H^{1}_{0}(\Omega))$ errors between  the exact solution $u$ and the  finite element solution $u_{h}$ are denoted by 
\begin{equation*}
\begin{aligned}
    \textbf{Error-1}&=\max_{1\leq n \leq N}\|u(t_{n})-u_{h}^{n}\|,\\
    \textbf{Error-2}&=\max_{1\leq n \leq N}\|\nabla u(t_{n})-\nabla u_{h}^{n}\|,
    \end{aligned}
\end{equation*}
respectively. The convergence rates are denoted by  \textbf{R-1} and \textbf{R-2}  with respect to \textbf{Error-1}  and  \textbf{Error-2}, respectively.\\

\noindent The following $\log$ vs. $\log$ formula is used to obtain the  convergence rates of the proposed numerical scheme \eqref{27-5-9a}-\eqref{27-5-9b}
\begin{equation*}
   \text{ Convergence rates} :=\begin{cases}\frac{\log(E(\tau,h_{1})/E(\tau,h_{2}))}{\log(h_{1}/h_{2})}& \text{in space direction},\\
   \frac{\log(E(\tau_{1},h)/E(\tau_{2},h))}{\log(\tau_{1}/\tau_{2})} & \text{in time direction},
\end{cases}
\end{equation*}
where $E(\tau,h_{1}),E(\tau,h_{2}),E(\tau_{1},h),E(\tau_{2},h)$ are the errors  at different mesh points.\\

\noindent 
 We consider the problem \eqref{22-5-1} in  the two dimensional domain  $\Omega=[0,1]\times[0,1]$ and $T=1$,  with the source function
 \begin{equation}\label{1-8-22}
    f(x,y,t)=\Gamma(1+\alpha)(x-x^{2})(y-y^{2})-\left(t^{\alpha}+\frac{t^{3\alpha}}{45}-\frac{t^{1+\alpha}}{1+\alpha}\right)2(x^{2}+y^{2}-x-y),
\end{equation}
where $(x,y)\in \Omega$ and $t\in [0,T].$ With this source function \eqref{1-8-22}, the exact solution of the problem  \eqref{22-5-1} is given by  $ u(x,t)=t^{\alpha}(x-x^{2})(y-y^{2})$
which exhibits the weak singularity at $t=0$ and satisfies the regularity assumption \eqref{28-6-1}. \\ 

\noindent \textbf{Errors and convergence rates in the space direction:} We fix the number of points in the time direction i.e $N=150$ so that spatial error dominates the convergence rates. We see  the spatial convergence rates  by running the MATLAB  code at different points (M) in the space direction.
\begin{table}[htbp]
	\centering
	\caption{\textbf{Errors and convergence rates in the space  direction for \boldmath$r=1$ }}
 	\begin{tabular}{|c|c|c|c|c|}
 	\hline
    &\multicolumn{2}{|c|}{\boldmath$\alpha=0.4$}&\multicolumn{2}{|c|}{\boldmath$\alpha=0.6$}\\
    \hline

		\textbf{M} & \textbf{Error-1 (R-1)} &\textbf{Error-2 (R-2)}  &  \textbf{Error-1 (R-1)}& \textbf{Error-2 (R-2)}\\
 		\hline

 		3 &1.88e-02 (-----) &  8.24e-02 (-----)   & 1.88e-02 (-----) & 8.25e-02 (-----)\\
 		
 		5 &5.71e-03 (1.724) & 4.38e-02 (0.911) & 5.71e-03 (1.724) & 4.38e-02 (0.911)\\
		
		9  &1.51e-03 (1.923) & 2.22e-02 (0.977)& 1.51e-03 (1.927)  & 2.22e-02 (0.977) \\
 		
 		17  &3.89e-04 (1.954)& 1.11e-02 (0.994)  & 3.83e-04 (1.973)& 1.11e-02 (0.994) \\
		\hline
 		
 	\textbf{Rate} & \quad $\simeq$~~~\textbf{2} &  \quad$\simeq$~~~\textbf{1}  & \quad$\simeq$~~~\textbf{2}   &   \quad $\simeq$~~~\textbf{1}\\
 		\hline
 	\end{tabular}%
 	\label{table1}%
 \end{table}%
 
 \noindent By observing the  Table \ref{table1}, we obtain the first order accuracy in the $L^{\infty}(0,T;H^{1}_{0}(\Omega))$ norm as we have predicted in Theorem \ref{9-7-2}. On the other hand it is interested to note the quadratic rate of convergence in the  $L^{\infty}(0,T;L^{2}(\Omega))$ norm which is higher than the accuracy rate obtained in Theorem \ref{9-7-2}. This type of phenomenon was also observed  in the literature \cite{gudi2012finite,kundu2016kirchhoff} for Kirchhoff type problems. At this point we are not able to prove the second order accuracy in $L^{\infty}(0,T;L^{2}(\Omega))$ norm which will be analysed in future.\\

 \noindent \textbf{Errors and convergence rates in the time direction:} 
 \noindent We collect the errors and the convergence rates in the time direction for the case $r=1$ i.e. uniform mesh in time by taking  $N=\lfloor M^{1/\alpha}\rfloor$ in Table \ref{table2}.
 \begin{table}[htbp]
	\centering
	\caption{\textbf{Errors and convergence rates in the \boldmath$L^{\infty}$ norm in the time direction for \boldmath$r=1$ }}
 	\begin{tabular}{|c|c|c|c|c|c|}
 	\hline
   & &\multicolumn{2}{|c|}{\boldmath$\alpha=0.4$}&\multicolumn{2}{|c|}{\boldmath$\alpha=0.6$}\\
    \hline

		\textbf{M} & \textbf{Error} & N&\textbf{Rate} & N& \textbf{Rate}\\
 		\hline

 		3 & 8.26e-02& 15 &  -----  & 6& ----- \\
 		
 		5 & 4.38e-03 & 55 & 0.468 & 14 & 0.661  \\
		
		9  &  2.22e-02 & 243 & 0.451  & 38 & 0.648 \\
 		
 		17  & 1.11e-02 & 1191 & 0.432 & 112 & 0.627 \\
 		
		\hline
		
		\textbf{Rate} &  & & $\simeq$~\textbf{0.4} & &  $\simeq$~\textbf{0.6}  \\
		
 		\hline
 	\end{tabular}%
 	\label{table2}%
 \end{table}%

 \noindent From Table \ref{table2}, we get the $\alpha$-convergence rate in the time direction as we have concluded in Theorem \ref{9-7-2}. Here we are getting less convergence rates in the time direction  because the solution of our problem \eqref{22-5-1} possess a weak singularity at $t=0$. To achieve the optimal convergence rate in the time direction for such type of solutions, we choose $r=2/\alpha$ and obtain the numerical results in Table \ref{table3} by setting $N=M$.\\
 
 \begin{table}[htbp]
	\centering
	\caption{\textbf{Errors and convergence rates in the \boldmath$L^{\infty}$ norm in the time direction for \boldmath$r=2/\alpha$ }}
 	\begin{tabular}{|c|c|c|c|c|}
 	\hline
    &\multicolumn{2}{|c|}{\boldmath$\alpha=0.4$}&\multicolumn{2}{|c|}{\boldmath$\alpha=0.6$}\\
    \hline

		\textbf{M=N} & \textbf{Error } &\textbf{Rate}  &  \textbf{Error}& \textbf{Rate}\\
 		\hline

 		3 & 2.21e-02  &  -----   & 2.06e-02  &  -----\\
 		
 		5 & 8.52e-03  & 1.371 & 7.27e-03  & 1.502\\
		
		9  & 2.77e-03 & 1.619 & 2.15e-03  & 1.755 \\
 		
 		17  & 8.06e-04  & 1.782  & 5.85e-04 & 1.881 \\
		\hline
		
 	\end{tabular}%
 	\label{table3}%
 \end{table}%
 \noindent In Table \ref{table3}, we see that the convergence rate is slow. This is because of the graded mesh on $[0,T]$. In the graded mesh on $[0,T]$, we have large time steps near final time $T$ and very small time steps near $t=0$ (see Figure 1).
 \begin{figure}[H]\label{6-8-2}
	\centering
\includegraphics[scale=0.6]{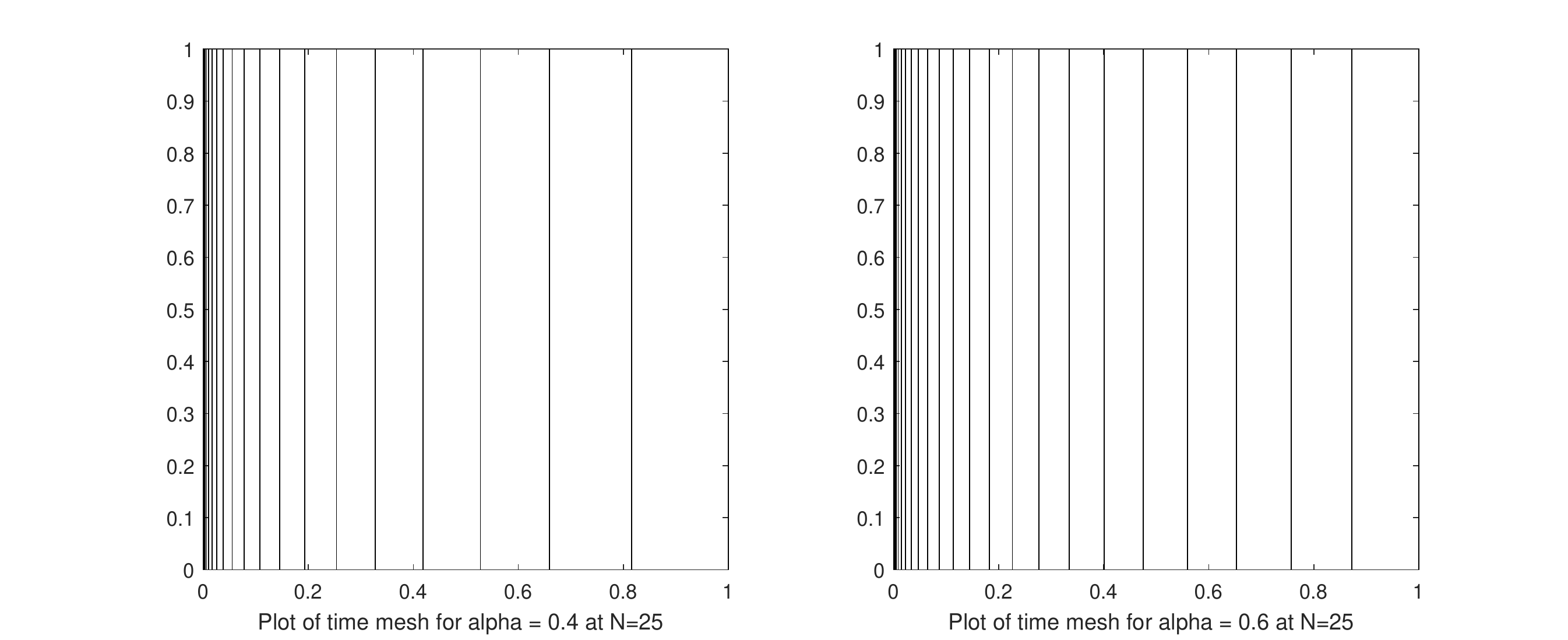}
\caption{ Graded mesh on $[0,T]$ for N=25.}
\end{figure}

\noindent This phenomena reduces the numerical resolution of the scheme near final time $T$. To overcome this difficulty, we divide the  time interval $[0,T]$ into two parts $[0,T_{0}]$ and $[T_{0},T]$ for sufficiently small $T_{0}=\min\{\frac{T}{2^{r}},(1-\frac{1}{r})T\}$ \cite{chen2019error}. We apply the graded mesh on the first part $[0,T_{0}]$ by taking the number of points $N_{0}=\lceil \rho_{0}N \rceil$ with $\rho_{0}=\min\{\frac{r}{2^{r}-1+r},\frac{r(r-1)}{1+r(r-1)}\}$ and uniform mesh on $[T_{0},T]$ with uniform time step $\tau_{0}=\frac{T-T_{0}}{N-N_{0}}$ (see Figure 2).  We choose these $T_{0},N_{0}$ and $\tau_{0}$ and collect the numerical results in Table \ref{table4} for $r=2/\alpha$.\\
 \begin{figure}[H]\label{6-8-3}
	\centering
\includegraphics[scale=0.6]{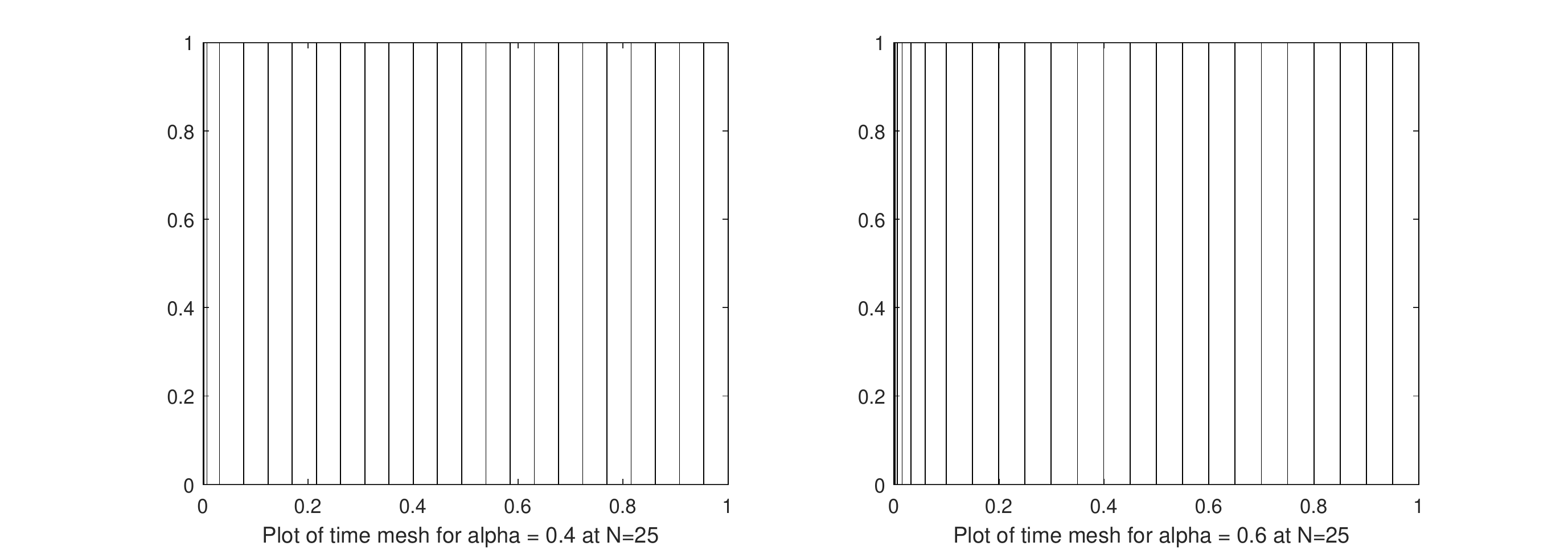}
\caption{ Graded mesh on $[0,T_{0}]$ and uniform mesh on $[T_{0},T]$ for N=25.}
\end{figure}

 \begin{table}[htbp]
	\centering
	\caption{\textbf{Errors and convergence rates in the \boldmath$L^{\infty}$ norm in the time direction for \boldmath$r=2/\alpha$ }}
 	\begin{tabular}{|c|c|c|c|c|}
 	\hline
    &\multicolumn{2}{|c|}{\boldmath$\alpha=0.4$}&\multicolumn{2}{|c|}{\boldmath$\alpha=0.6$}\\
    \hline

		\textbf{M=N} & \textbf{Error } &\textbf{Rate}  &  \textbf{Error}& \textbf{Rate}\\
 		\hline

 		3 & 2.07e-02  &  -----   & 1.96e-02  &  -----\\
 		
 		5 & 6.76e-03  & 1.619 & 6.49e-03  & 1.596\\
		
		9  & 1.91e-03 & 1.826 & 1.73e-03  & 1.906 \\
 		
 		17  & 4.75e-04  & 2.005  & 4.41e-04 & 1.973  \\
		\hline
		
			\textbf{Rate} &  &$\simeq$~~\textbf{2} & &  $\simeq$~~\textbf{2}   \\
 		\hline
 	\end{tabular}%
 	\label{table4}%
 \end{table}%
 \newpage
 \noindent In Table \ref{table4}, we observe the  optimal convergence rate in the time direction faster than the previous one. These results confirm the theoretical estimates as  we have proved in Theorem \ref{9-7-2}. One can choose $r > \frac{2}{\alpha}$  to obtain the second order accuracy in time but with a constant factor that grows with $r$. Hence $r = \frac{2}{\alpha}$ is the optimal grading parameter in the time graded mesh which gives the optimal   rate of convergence for the non-smooth solutions of  the problem \eqref{22-5-1}.\\ 
 
 \noindent Now, we plot the graph of an approximate solution as well as an exact solution in Figure 3 using numerical scheme \eqref{27-5-9a}-\eqref{27-5-9b} for  $\alpha=0.5$.
 
\begin{figure}[H]\label{6-8-4}
	\centering
\includegraphics[scale=0.5]{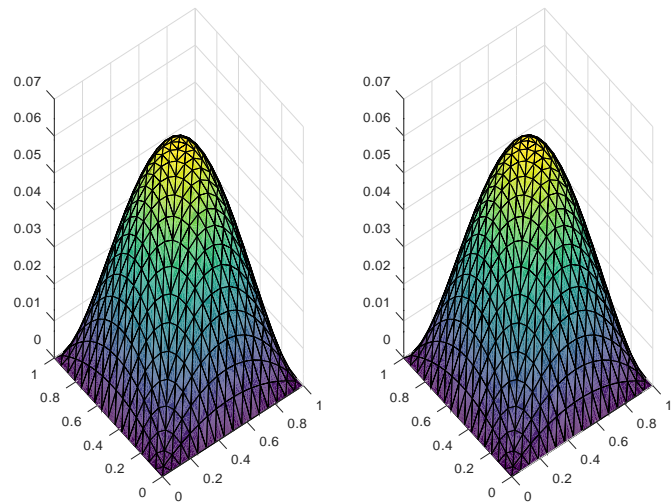}
\caption{Approximate solution (L.H.S) and Exact solution (R.H.S) at $T=1$  and $\alpha =0.5$.}
\end{figure}

\section{Concluding remarks}
In this work, we  developed a linearized fractional Crank-Nicolson-Galerkin scheme \eqref{27-5-9a}-\eqref{27-5-9b} which reduces the computational cost in comparison to the Newton-Raphson method. Using a weighted $H^{1}(\Omega)$ norm \eqref{22-6-4}, we derived a priori bounds \eqref{1-7-1} and convergence estimates \eqref{8-7-1z}-\eqref{8-7-1y} for the solutions of the proposed numerical scheme. We proved that the convergence rates are of $O\left(M^{-1}+N^{-2}\right)$ in $L^{\infty}(0,T;L^{2}(\Omega))$ norm as well as in $L^{\infty}(0,T;H^{1}_{0}(\Omega))$ norm. Finally we  offered  the numerical results which perfectly support the sharpness of the theoretical claims.

\bibliographystyle{plain}
\bibliography{FP3BIB}

\end{document}